\documentclass[14pt, dvipsnames]{extarticle}
\usepackage[left=2cm, top=2cm, right=2cm, bottom=20mm]{geometry} 

\usepackage[T1]{fontenc}
\usepackage[cmintegrals,cmbraces]{newtxmath}
\usepackage[lining]{ebgaramond}
\usepackage{ebgaramond-maths}
\usepackage{euscript}
\usepackage{xcolor}
\usepackage{graphicx}
\usepackage[figurename=Fig.]{caption}

\usepackage[tracking = true, letterspace = 50]{microtype}
\SetTracking{encoding=T1, shape=sc}{50}
\SetTracking{encoding=T1}{50}

\usepackage{amsthm}
\usepackage{csquotes}
\usepackage{hyperref}
\usepackage{xurl}

\usepackage[
 backend=biber,
 style=alphabetic,
 sorting=nyt,
 giveninits=true,
 maxbibnames=99,
 doi=true,
 url=true,
 isbn=false,
 eprint=true
]{biblatex}
\AtBeginBibliography{\normalsize\raggedright}
\renewbibmacro{in:}{}

\usepackage[explicit]{titlesec}
\usepackage{pgfornament}

\newcommand{\sectionflourish}{%
  \pgfornament[width=3cm]{87}%
}

\titleformat{\section}[block]
  {\normalfont\Large\bfseries\filcenter} 
  {}
  {0pt}
  {\thesection.\hspace{0.45em}#1}
  [\vspace{-0.15em}\sectionflourish]

\titleformat{name=\section,numberless}[block]
  {\normalfont\Large\bfseries\filcenter}
  {}
  {0pt}
  {#1}
  [\vspace{-0.15em}\sectionflourish]

\titlespacing*{\section}
  {0pt}
  {3ex plus 0.3ex minus 0.2ex}
  {2.5ex}

\theoremstyle{plain}

\newtheorem{theorem}{Theorem}[section]
\newtheorem{prop}[theorem]{Proposition}
\newtheorem{lemma}[theorem]{Lemma}
\newtheorem{coroll}[theorem]{Corollary}

\newtheorem*{remark*}{Remark}

\theoremstyle{definition}

\newtheorem{defi}[theorem]{Definition}

\newtheorem{remark}[theorem]{Remark}

\usepackage{enumerate}

\usepackage{etoolbox}
\patchcmd{\thmhead}{(#3)}{#3}{}{}
\makeatletter
\renewenvironment{proof}[1][\proofname]{
  \par\pushQED{\qed}\normalfont
  \topsep6\p@\@plus6\p@\relax
  \trivlist\item[\hskip\labelsep\bfseries\itshape #1\@addpunct{.}]
}{\popQED\endtrivlist\@endpefalse}
\makeatother

\let\phi\varphi

\DeclareMathOperator{\Bb}{\EuScript{B}}

\DeclareMathOperator{\pt}{\mathrm{pt}}

\DeclareMathOperator{\Sss}{\mathbb{S}}

\DeclareMathOperator{\disc}{\mathrm{disc}}

\DeclareMathOperator{\pr}{\mathrm{pr}}

\DeclareMathOperator{\Q}{\mathbb{Q}}

\DeclareMathOperator{\Ker}{\mathrm{Ker}}

\DeclareMathOperator{\G}{\mathbb{G}}

\DeclareMathOperator{\inv}{\mathrm{inv}}

\DeclareMathOperator{\Z}{\mathbb{Z}}
\DeclareMathOperator{\Cc}{\mathbb{C}}

\DeclareMathOperator{\CP}{\mathbb{C}P}
\DeclareMathOperator{\Hhh}{\mathrm{H}}

\newcommand{\dd}{\mathrm{d}}

\DeclareMathOperator{\Sp}{\mathrm{Sp}}
\DeclareMathOperator{\Uu}{\mathrm{U}}

\DeclareMathOperator{\MSp}{\mathrm{MSp}}

\DeclareMathOperator{\MU}{\mathrm{MU}}

\DeclareMathOperator{\HP}{\mathbb{H}\mathit{P}}

\DeclareMathOperator{\SO}{\mathrm{SO}}

\DeclareMathOperator{\Kr}{\mathrm{Kr}}

\DeclareMathOperator{\wpp}{\tilde{\wp}}
\DeclareMathOperator{\Bc}{\mathrm{Bc}}
\DeclareMathOperator{\Buc}{\mathrm{Buc}}
\DeclareMathOperator{\Ff}{\EuScript{F}}

\DeclareMathOperator{\Pic}{\mathrm{Pic}}

\DeclareMathOperator{\ch}{\mathrm{ch}}
\DeclareMathOperator{\Wt}{\mathrm{Wt}}
\DeclareMathOperator{\Lll}{\EuScript{L}}

\makeatletter
  \DeclareSymbolFont{ntxletters}{OML}{ntxmi}{m}{it}
  \SetSymbolFont{ntxletters}{bold}{OML}{ntxmi}{b}{it}
  \re@DeclareMathSymbol{\leftharpoonup}{\mathrel}{ntxletters}{"28}
  \re@DeclareMathSymbol{\leftharpoondown}{\mathrel}{ntxletters}{"29}
  \re@DeclareMathSymbol{\rightharpoonup}{\mathrel}{ntxletters}{"2A}
  \re@DeclareMathSymbol{\rightharpoondown}{\mathrel}{ntxletters}{"2B}
  \re@DeclareMathSymbol{\triangleleft}{\mathbin}{ntxletters}{"2F}
  \re@DeclareMathSymbol{\triangleright}{\mathbin}{ntxletters}{"2E}
  \re@DeclareMathSymbol{\partial}{\mathord}{ntxletters}{"40}
  \re@DeclareMathSymbol{\flat}{\mathord}{ntxletters}{"5B}
  \re@DeclareMathSymbol{\natural}{\mathord}{ntxletters}{"5C}
  \re@DeclareMathSymbol{\star}{\mathbin}{ntxletters}{"3F}
  \re@DeclareMathSymbol{\smile}{\mathrel}{ntxletters}{"5E}
  \re@DeclareMathSymbol{\frown}{\mathrel}{ntxletters}{"5F}
  \re@DeclareMathSymbol{\sharp}{\mathord}{ntxletters}{"5D}
  \re@DeclareMathAccent{\vec}{\mathord}{ntxletters}{"7E}
  \DeclareSymbolFont{cmletters}{OML}{cmm}{m}{it}
  \SetSymbolFont{cmletters}{bold}{OML}{cmm}{b}{it}
  \re@DeclareMathSymbol{\wp}{\mathord}{cmletters}{"7D}
\makeatother

\renewcommand{\epsilon}{\varepsilon}

\DeclareFontFamily{U}{EBSUB}{}
\DeclareFontShape{U}{EBSUB}{m}{it}{<-> EBGaramond-Italic-tlf-lgr}{}
\DeclareFontFamily{U}{EBSUB}{}
\DeclareSymbolFont{EBSUB}{U}{EBSUB}{m}{it}

\DeclareMathSymbol{\mu}{\mathord}{EBSUB}{`m}
\DeclareMathSymbol{\varDelta}{\mathord}{EBSUB}{`D}
\DeclareMathSymbol{\varOmega}{\mathord}{EBSUB}{`W}

 \hypersetup{
    colorlinks,
    linkcolor={red!50!black},
    citecolor={blue!50!black},
    urlcolor={blue!80!black}
}

\usepackage{stackrel}
\usepackage[all]{xy}

\usepackage{bm}

\usepackage{indentfirst}

\newcommand{\pullbackcorner}[1][dr]{\save*!/#1-1.7pc/#1:(-1,1)@^{|-}\restore}

\usepackage{tocloft}

\title{\bf Buchstaber, Ochanine,\\ Krichever, and Witten Genera}

\date{}

\author{Mikhail Kornev}

\begin{document}

\maketitle

\begin{abstract}
We introduce a new class of one-dimensional commutative formal group
laws with formal inverse $\bar u=-u$ whose modulus square construction
yields Buchstaber's family of polynomials, and prove that the formal
group law $F_{\mathrm{Bc}}(u,v)$ is universal for this class over
commutative $\mathbb{Z}[1/2]$-algebras. This class is related to, but
does not coincide with, the family of formal group laws associated
with the Krichever genus. We compute the values of the corresponding
Hirzebruch genus $\mathrm{Bc}$ on theta divisors and complex projective
spaces, describe its relation to the Ochanine, Krichever, and Witten
genera, and show how this construction gives examples not arising
from Hirzebruch's elliptic genera of level $n$. We construct a
complex-oriented multiplicative cohomology theory $\mathrm{Buc}^*$
such that the complex orientation $\mathrm{U}^*\to\mathrm{Buc}^*$ induces the
genus $\mathrm{Bc}$ on coefficient rings, and prove that its
localization at the discriminant is naturally isomorphic to the
Landweber-exact elliptic cohomology theory
$\mathrm{U}^*(-)\otimes_{\Omega_{\mathrm{U}}}\mathbb{Z}[1/2,a_1,a_2,a_3,\Delta^{-1}]$.
Finally, we prove that $\mathbb{Z}[a_1,12a_2,360a_3]$ is the smallest
subring of $\mathbb{Q}[a_1,a_2,a_3]$ over which the Buchstaber
exponential $f_{\mathrm{Bc}}(u)$ is a Hurwitz series. As a corollary,
we prove the Hurwitz-integrality statement predicted by Bunkova's
coefficient-divisibility conjecture and show that
$\mathbb{Z}[g_2/2,6g_3]$ is the minimal Hurwitz coefficient ring
of the Weierstrass sigma-function.
\end{abstract}

\date{}

\begin{center}
  \pgfornament[width=3cm]{87}
\end{center}

\tableofcontents

\thispagestyle{empty}

\newpage

\section{Introduction}

In \cite{Buchstaber90}, a family of polynomials
\begin{equation}\label{B_a_polynomial}
\Bb_{\boldsymbol{a}}(z; x, y) = (x + y + z - a_2xyz)^2 - 4(1 + a_3xyz)(xy + yz + xz + a_1xyz),
\end{equation}
where $\bm{a} = (a_1, a_2, a_3)$, was introduced\footnote{In \cite{Buchstaber90}, the polynomial \eqref{B_a_polynomial} is written using a different normalization $\widetilde{a}_1$, $\widetilde{a}_2$, and $\widetilde{a}_3$ of the free parameters $a_1$, $a_2$, and $a_3$. More precisely,
\[
\widetilde{a}_1=a_1,\qquad
\widetilde{a}_2=a_2,\qquad
\widetilde{a}_3=4a_3.
\]
This difference in normalization is immaterial, however, since the arguments are carried out over $\Z\left[\frac12\right]$-algebras.}. It was shown that each polynomial defines a two-valued formal group $\G_{\bm{a}}$, universal in the class of two-valued formal groups obtained by the modulus square construction, see \cite{Buchstaber_Novikov}, from the class of formal groups
\begin{equation}\label{the_class_F}
\Ff(u, v) = \frac{u^2\lambda_1(v)-v^2\lambda_1(u)}{u\lambda_2(v)-v\lambda_2(u)},
\end{equation}
where $\lambda_k(u)\in A[[u]]$, $k=1,2$; the graded ring $A$ has a rather complicated structure, see \cite{Buchstaber_Ustinov15}, but
\[
A\otimes \mathbb Q \cong \mathbb Q[b_1,b_2,b_3,b_4],
\qquad
b_1 = \alpha,\quad b_2 = \wp(z),\quad b_3 = \wp'(z),\quad b_4 = g_2,
\]
with degrees $\deg b_i = -2i$, $i=1,\ldots,4$, and $\wp(z) = \wp(z; g_2, g_3)$ is the Weierstrass function satisfying the equation
\[
(\wp'(z))^2 = 4\wp(z)^3 - g_2\wp(z)-g_3
\]
($z$ is a parameter of the uniformization of an elliptic curve, and $\alpha$ is a point on the curve).

It was shown in \cite[Theorem 6.3]{Buchstaber_Veselov19}, \cite[Theorem 8]{BK} that for free complex parameters $\bm{a}=(a_1,a_2,a_3)$ the Buchstaber polynomial \eqref{B_a_polynomial} defines the structure of the universal symmetric $2$-algebraic $2$-valued group $\G_{\Cc}(\Bb_{\bm{a}})$ on $\Cc$ with addition
\[
x\ast y=\{z\mid \Bb_{\bm{a}}(z;x,y)=0\},
\]
zero $0$, and inverse $\inv(x)=x$.

The class of groups \eqref{the_class_F} contains the formal groups corresponding to the Todd genus, the signature, the Ochanine (also elliptic) genus \cite{Ochanine87}, and the Krichever genus. In \cite{Buchstaber90}, it was proved that the class $\Ff(u,v)$ is a realization of the class of Krichever formal groups, see \cite{Krichever90}, with exponential
\begin{equation}\label{f_Krichever}
f_{\Kr}(x)=\frac{e^{\alpha x}}{\Phi(x, z)},
\qquad
\text{where}\qquad
\Phi(x, z) = \Phi(x, z; g_2, g_3) = \frac{\sigma(z-x)}{\sigma(z)\sigma(x)}e^{\zeta(z)x},
\end{equation}
is the Baker--Akhiezer function, and $\sigma(z) = \sigma(z; g_2, g_3)$, $\zeta(z) = \frac{d}{dz}\log\sigma(z)$ are Weierstrass functions, see, for example, \cite{Lang87}, \cite{WW96}.

In the present paper, a new class of formal groups
\begin{equation}\label{Buchstaber_formal_group}
F_{\Bc}(u,v)= \left[\left(\frac{u\sqrt{Q(v)}-v\sqrt{Q(u)}}{u^2-v^2}\right)^2+a_3u^2v^2\right]^{-1/2},
\end{equation}
where $Q(t) = 1-a_1t^2+a_2t^4-a_3t^6$ is introduced ($a_1$, $a_2$, and $a_3$ are free parameters), whose modulus square is given by the family of polynomials \eqref{B_a_polynomial}. We prove that, over commutative \(\mathbb Z[1/2]\)-algebras, the Buchstaber formal group law is universal among formal group laws whose modulus square construction yields the corresponding Buchstaber two-valued formal group $\mathbb G_{\boldsymbol{a}}$. Since $F_{\Bc}(u,v)$ is a formal group law over the graded ring $$R_{\Bc} := \Z\left[\frac12, a_1, a_2, a_3 \right],\quad\deg a_i=-4i,\qquad i=1,2,3,$$ the
universality of the formal group law of complex cobordism determines
a unique Hirzebruch genus
\[
\Bc\colon \Omega_{\Uu}\longrightarrow R_{\Bc},
\]
which we call the \emph{Buchstaber genus}. Here
$\Omega_{\Uu}$ denotes the complex cobordism ring. 

The formal group law $F_{\Bc}(u,v)$ belongs to the family of formal group laws \eqref{the_class_F} if and only if $a_3=0$. In the case of $a_3=0$, we get the Ochanine genus. Therefore, for $a_3\neq 0$, the corresponding genus is distinct from the elliptic genera of level $n$ introduced by Hirzebruch, see \cite[Appendix III, Section 1]{Hirzebruch88}.

At the cohomological level, we construct a complex-oriented multiplicative theory whose coefficient homomorphism is the Buchstaber genus. More precisely, we construct a complex-oriented multiplicative cohomology theory \(\mathrm{Buc}^*\) such that the orientation homomorphism\footnote{Here and below, $\Uu^{\ast}$ and $\Sp^{\ast}$ denote the
cohomology theories represented by the spectra $\MU$ and $\MSp$,
respectively.} $\Uu^*\to\mathrm{Buc}^*$ induces the Buchstaber genus on coefficient rings, and prove that its discriminant localization is naturally isomorphic to the Landweber-exact elliptic theory $\Uu^*(-)\otimes_{\Omega_{\Uu}}\mathbb Z\!\left[\frac12,a_1,a_2,a_3,\Delta^{-1}\right]$, where 
\begin{equation*}
\begin{aligned}
\Delta
&=\disc_t\bigl(t^3-a_1t^2+a_2t-a_3\bigr)\\
&=a_1^2a_2^2-4a_1^3a_3-4a_2^3+18a_1a_2a_3-27a_3^2.
\end{aligned}
\end{equation*} The construction uses oriented bordism with singularities after inverting $2$, in the spirit of the Baas–Sullivan theory  \cite{Baas73, LRS95} and Mironov's results \cite{Mironov79}.

In 1882, Weierstrass obtained a recursive expansion of the elliptic sigma-function and proved that its normalized coefficients are integers \cite{Weierstrass1882}. Equivalently, in modern Hurwitz-series notation (see \eqref{Hurwitz_series_definition} for the definition),
\[
\sigma(u;g_2,g_3)\in\mathcal H_{\Z[g_2/2,\,2g_3]}[[u]].
\]
Buchstaber and Leikin later recovered this recurrence from their heat-equation formalism and extended the method to sigma-functions of higher-genus algebraic curves \cite{Buchstaber_Leikin04,Buchstaber_Leikin05}. In \cite{Bunkova17}, Bunkova formulated a coefficient-divisibility conjecture giving exact $2$-adic and $3$-adic valuation formulas for the normalized coefficients of the sigma-function. In particular, this conjecture predicts that
\[
\sigma(u;g_2,g_3)\in\mathcal H_{\Z[g_2/2,\,6g_3]}[[u]].
\]
We do not prove the stronger valuation formulas here. Instead, we prove that the full three-parameter Buchstaber exponential satisfies
\[
f_{\Bc}(u)\in\mathcal H_{\Z[a_1,12a_2,360a_3]}[[u]]
\]
and that $\Z[a_1,12a_2,360a_3]$ is its minimal Hurwitz coefficient ring. Using the specialization $a_1=0$, we establish the Hurwitz-integrality consequence of Bunkova's conjecture for the Weierstrass sigma-function and determine its minimal Hurwitz coefficient ring.

The paper is organized as follows. 

In Section \ref{modulus_square_construction}, we recall the basics of the theory of two-valued formal groups. Section \ref{The_Buchstaber_Genus} is devoted to Theorem \ref{Buchstaber_genus_theorem}. Corollary \ref{coroll_Buchstaber_genus} says that the formal group \eqref{Buchstaber_formal_group} defines a Hirzebruch genus $\Bc$ called the Buchstaber genus, while Corollary~\ref{coroll_integral_universality} proves that, over commutative $\Z[1/2]$-algebras, every one-dimensional commutative formal group law with formal inverse $\bar u=-u$
whose modulus square belongs to the Buchstaber family is obtained from $F_{\Bc}$ by a unique specialization of the parameters. As applications, we compute its values $\Bc(\Theta_n)$ on the smooth theta divisors $\Theta_n$ and complex projective spaces. It turns out that the values $\Bc(\Theta_n)$ belong to the ring $\Z[a_1,a_2,a_3]$, see Proposition \ref{Bc_of_Theta_is_in_Z}. In Proposition \ref{element_K}, we compare the Buchstaber, Ochanine, and Krichever genera. In Section \ref{Buchstaber_cohomology_theory}, we construct the Buchstaber cohomology theory and identify its localization at the discriminant; see Theorems~\ref{Buc_spectrum} and~\ref{Buchstaber_spectrum_is_Landweber_exact}. Section \ref{Bunkova_conjecture} compares the Buchstaber and Witten exponentials, determines the minimal Hurwitz coefficient ring of $f_{\Bc}$ (see Theorem~\ref{Buchstaber_exponential_Hurwitz_theorem}), and derives the Hurwitz-integrality consequence of Bunkova's conjecture (see Corollary~\ref{Bunkova_Hurwitz_prediction}).

The present paper is a substantially expanded version of \cite{Kornev26}, incorporating the $\Z[1/2]$-universality theorem for the Buchstaber formal group law, the construction of the complex-oriented multiplicative cohomology theory, and a proof of the Hurwitz-integrality consequence of Bunkova's conjecture.

The author is grateful to Victor M. Buchstaber and Sergei K. Lando for useful discussions of this work.

\section{The Modulus Square Construction}\label{modulus_square_construction}

In \cite{Buchstaber_Novikov}, a modulus square construction and the corresponding two-valued formal group in the complex cobordism ring $\Omega_{\Uu}$ were introduced. The general theory of two-valued formal groups was developed in \cite{Buchstaber75}. For the definition, examples, and properties of $n$-valued groups, see the survey \cite{Buchstaber}, or recent works \cite{BK}. Let us recall the modulus square construction. 

Let $\zeta_1 = \pr_1^{\ast}(\zeta)$ and $\zeta_2 = \pr_2^{\ast}(\zeta)$ be quaternionic line bundles, where $\zeta$ is the universal quaternionic line bundle over $\HP^{\infty}$, and $\pr_j:\HP^{\infty}\times\HP^{\infty}\to \HP^{\infty}$ are two projections onto the Cartesian factors. The embedding $\Sss^1\rightarrow\Sp(1)$ of the multiplicative circle group $\Sss^1$ into the group of unit quaternions $\Sp(1)$ induces a map
\begin{equation}\label{iota_map}
\iota:\Cc\!P^{\infty}\to\HP^{\infty}
\end{equation}
between their classifying spaces. The pullback of the bundle $\zeta_j$ ($j=1,2$) under the map $\iota$ decomposes as a sum $\eta_j\oplus\overline{\eta}_j$, where $\eta_j$ is the universal complex line bundle. We have the following pullback diagram:
$$
\xymatrix{
(\eta_1\oplus\overline{\eta}_1)\otimes_{\Cc}(\eta_2\oplus\overline{\eta}_2)\pullbackcorner\ar[rr]\ar[d] & & \zeta_1\otimes_{\Cc}\zeta_2\ar[d]\\
\CP^{\infty}\times\CP^{\infty}\ar[rr]^{\iota\times \iota} & & \HP^{\infty}\times\HP^{\infty}
}
$$

There is an isomorphism of complex vector bundles: $$(\eta_1\oplus\overline{\eta}_1)\otimes_{\Cc}(\eta_2\oplus\overline{\eta}_2) \cong (\eta_1\eta_2\oplus\overline{\eta}_1\overline{\eta}_2)\oplus(\eta_1\overline{\eta}_2\oplus\overline{\eta}_1\eta_2).$$ Note that the bundles $\xi_1:= \eta_1\eta_2\oplus\overline{\eta}_1\overline{\eta}_2$ and $\xi_2 :=\eta_1\overline{\eta}_2\oplus\overline{\eta}_1\eta_2$ admit quaternionic structures. 

Let $\Pic(M)$ be the Picard group of isomorphism classes $[\xi]$ of complex linear bundles over a complex manifold $M$. Consider an involution $\sigma: [\xi]\mapsto[\overline{\xi}]$. By abuse of notation, we will write $\xi$ instead of $[\xi]$. The points of the corresponding orbit space $X:=\Pic(M)/\sigma$ are identified with unordered pairs $[\xi, \overline{\xi}]$. We get the 2-valued coset group (see \cite[Section 6]{Buchstaber} for the definition) with multiplication $$[\xi,\overline{\xi}]\ast[\eta,\overline{\eta}] = [[\xi\eta,\overline{\xi}\overline{\eta}], [\xi\overline{\eta}, \overline{\xi}\eta]]$$ and neutral element $[\mathbf{1}_{\Cc}, \mathbf{1}_{\Cc}]$, where $\mathbf{1}_{\Cc}$ denotes a class of a trivial complex line bundle.

Note that $\zeta_1\otimes_{\Cc}\zeta_2$ cannot admit a quaternionic structure (an exercise in representation theory). But the bundle $\zeta_1\otimes_{\Cc}\zeta_2\otimes_{\Cc}\zeta_3$ over $\HP^{\infty}\times\HP^{\infty}\times\HP^{\infty}$ does admit one. A natural isomorphism $$(\zeta_1\otimes_{\Cc}\zeta_2)\otimes_{\Cc}\zeta_3\cong \zeta_1\otimes_{\Cc}(\zeta_2\otimes_{\Cc}\zeta_3)$$ is equivalent to the associativity of the two-valued operation $\ast$. 

This two-valued group is a geometric realization of a more general construction: 

\begin{prop}\label{modulus_square_prop}
Let $G$ be an abelian group and $\sigma: g\mapsto -g$ an involution. Denote by $X$ the orbit space $G/\sigma$ with points $[g, -g]$. Then $X$ carries an involutive commutative two-valued coset group structure with operation
$$
[g, -g]\ast[h,-h] = [[g+h, -g-h], [g-h, -g+h]]
$$ with neutral element $[e, e]$, and $e$ is a unit.
\end{prop}

\begin{proof}
 This follows directly from the coset construction, see \cite[Section 6]{Buchstaber}.
\end{proof}

\begin{defi}
The two-valued group considered in Proposition \ref{modulus_square_prop} is called the {\it modulus square construction} for an abelian group $G$. 
\end{defi}

The modulus square construction has the following infinitesimal analogue. Let $F(u, v)$ be a commutative formal group over an algebra $A$ with logarithm $g(u)$ and exponential $f(u)$ (they are defined over $A\otimes \Q$). Let $x = u\overline{u}$, $y = v\overline{v}$. Then, by definition, a {\it two-valued formal group} $\G_F$ is given by the law \begin{equation}\label{two-valued_law}x\ast y = [F(u,v) F(\overline{u}, \overline{v}), F(u, \overline{v})F(\overline{u}, v)],\end{equation} neutral element $x = 0$ and inverse $\inv(x) = x$.

Define the following formal series: \begin{equation}\label{Theta_1_and_Theta_2_z_1_and_z_2}\begin{matrix}z_1 = F(u, v)F(\overline{u}, \overline{v}), & & z_2 = F(u,\overline{v})F(\overline{u}, v)\\ \Psi_1 = z_1 + z_2, & & \Psi_2 = z_1z_2. \end{matrix}\end{equation} One can check that these series belong to the ring $A[[x, y]]$, i.e. $\Psi_1 = \Psi_1(x, y)$ and $\Psi_2 = \Psi_2(x, y)$, see \cite[Lemma 2.21]{Buchstaber_Novikov}. Hence, we get the following

\begin{prop}\label{two_valued_formal_group_roots}
The two-valued formal group $\G_F$ is determined by the roots $z_1,z_2$ of a quadratic polynomial over the ring $A[[x,y]]$ (see \eqref{Theta_1_and_Theta_2_z_1_and_z_2} for the notation):
\begin{equation}\label{Theta_1_and_Theta_2_general}
\begin{aligned}
z^2-\Psi_1(x,y)z+\Psi_2(x,y)&=0,\\
x\ast y&=[z_1,z_2].
\end{aligned}
\end{equation}
\end{prop} 

A general definition of a multi-valued formal group was given in \cite[Section 1]{Buchstaber75}.

\begin{defi}
In the notation of \eqref{Theta_1_and_Theta_2_z_1_and_z_2}, the series \begin{equation}\label{logarithm_of_two-valued_group_2}B(x) = g(u)g(\overline{u}) = -g(u)^2\end{equation} is called a logarithm of the two-valued formal group $\G_F$ with the multiplication \eqref{two-valued_law}.
\end{defi}

The series $B(x)$ in \cite{Buchstaber_Novikov} was called the logarithm for the following reason: 

\begin{prop}\label{z_1_2}
In the notations of \eqref{Theta_1_and_Theta_2_z_1_and_z_2}$\ and\ \eqref{logarithm_of_two-valued_group_2}$$:$
$$
z_{1,2} = B^{-1}\left(\left(\sqrt{B\left(x\right)} \pm \sqrt{B\left(y\right)}\right)^2\right).
$$
\end{prop}

\begin{proof}
We have
$$
B^{-1}((\sqrt{B(x)} \pm \sqrt{B(y)})^2) = B^{-1}(-(g(u)\pm g(v))^2).
$$
It is enough to check that
$$
B(z_{1,2}) = -(g(u)\pm g(v))^2.
$$
Consider the case of $z_1$ (the case of $z_2$ is completely analogous). We have
$$
B(z_1) = B(V\overline{V}) = -g(V)^2 = -(gg^{-1}(g(u) + g(v)))^2 = -(g(u) + g(v))^2,
$$
where $V = g^{-1}(g(u) + g(v))$. This is exactly what was required.
\end{proof}

Recall the topological interpretation of the two-valued formal group
$\G_{F_{\Uu}}$ developed in \cite{Buchstaber_Novikov}. Let
\[
\rho\colon \Sp^*(-)\longrightarrow \Uu^*(-)
\]
be the natural transformation obtained by forgetting the symplectic
structure, and put
\[
j:=\iota\times\iota\colon
\CP^\infty\times\CP^\infty
\longrightarrow
\HP^\infty\times\HP^\infty.
\]

Let $\eta_i=\pr_i^*(\eta)$, where $\eta$ is the universal complex line
bundle over $\CP^\infty$, and put
\[
u=c_1^{\Uu}(\eta_1),\qquad
v=c_1^{\Uu}(\eta_2),\qquad
\bar u=c_1^{\Uu}(\bar\eta_1),\qquad
\bar v=c_1^{\Uu}(\bar\eta_2).
\]
As we already know, the pullback of $\zeta_1\otimes_{\Cc}\zeta_2$ decomposes as
\[
j^*(\zeta_1\otimes_{\Cc}\zeta_2)
\cong
\xi_1\oplus\xi_2,
\]
where both $\xi_1$ and $\xi_2$ admit quaternionic structures.

Introduce the classes
\[
\begin{aligned}
z_1
&:=c_2^{\Uu}(\xi_1)
 =F_{\Uu}(u,v)F_{\Uu}(\bar u,\bar v),\\
z_2
&:=c_2^{\Uu}(\xi_2)
 =F_{\Uu}(u,\bar v)F_{\Uu}(\bar u,v).
\end{aligned}
\]
Furthermore, put
\[
\begin{aligned}
x
&:=\rho\!\left(\, j^*p_1^{\Sp}(\zeta_1)\right)
  =u\bar u,\\
y
&:=\rho\!\left(\, j^*p_1^{\Sp}(\zeta_2)\right)
  =v\bar v.
\end{aligned}
\]
Thus
\[
x,y\in
\Uu^4(\CP^\infty\times\CP^\infty).
\]

The two symmetric functions of the values $z_1,z_2$ are
\[
\begin{aligned}
\Psi_1(x,y)
&:=
\rho\!\left(
p_1^{\Sp}\bigl(j^*(\zeta_1\otimes_{\Cc}\zeta_2)\bigr)
\right)
=z_1+z_2,\\
\Psi_2(x,y)
&:=
\rho\!\left(
p_2^{\Sp}\bigl(j^*(\zeta_1\otimes_{\Cc}\zeta_2)\bigr)
\right)
=z_1z_2.
\end{aligned}
\]
Consequently, $z_1$ and $z_2$ are the roots of
\[
z^2-\Psi_1(x,y)z+\Psi_2(x,y)=0,
\]
and hence determine the two-valued formal group
$\G_{F_{\Uu}}$.

Now let
\[
X:=\rho\!\left(p_1^{\Sp}(\zeta)\right)
   =c_2^{\Uu}(\zeta_{\Cc})
   \in\Uu^4(\HP^\infty),
\]
where $\zeta_{\Cc}$ denotes the underlying complex bundle of the
universal quaternionic line bundle $\zeta$. Then
\[
\iota^*X
=
c_1^{\Uu}(\eta)c_1^{\Uu}(\bar\eta)
=
u\bar u
=
x.
\]

Let $z\in H^4(\HP^\infty;\Z)$ be the generator chosen so that
\[
\iota^*z=-t^2,
\]
where $t\in H^2(\CP^\infty;\Z)$ is the standard generator. By
Buchstaber and Novikov \cite[page~93]{Buchstaber_Novikov},
\[
\ch_{\Uu}(g_{\Uu}(u))=t,
\qquad
\ch_{\Uu}(B(X))=z.
\]
Equivalently,
\[
\ch_{\Uu}(X)=B^{-1}(z).
\]
After applying $\iota^*$, this identity becomes
\[
\ch_{\Uu}(B(x))
=
\iota^*z
=
-t^2,
\]
in agreement with
\[
B(x)=B(u\bar u)=-g_{\Uu}(u)^2.
\]

\vspace{1.5em}

The notions of logarithm $B(x)$ and exponential $B^{-1}(x)$ can be generalized to the general case of 2-valued formal groups. All the 2-valued formal groups of the form \eqref{Theta_1_and_Theta_2_general} were classified by Buchstaber in \cite{Buchstaber75}. This classification shows that every commutative one-dimensional two-valued formal group $\mathbb G(R)$ over a commutative $\mathbb{Q}$-algebra $R$ that is an integral domain is strongly isomorphic either to the elementary two-valued formal group of the first type,
\[
z^{2}-2(x+y)z+(x-y)^{2}=0,
\]or to the two-valued formal group of the second type,
\[
(z-x-y)^{2}=0,
\]which is the square of the additive formal group. The first case is the one relevant below. For a two-valued formal group \(\mathbb G(R)\) of the first type, its {\it logarithm} is the unique normalized series \(B(x)=x+O(x^{2})\) that establishes this strong isomorphism. Equivalently, if \(z_{1}\) and \(z_{2}\) are the two values of \(x*y\), then
\[
B(z_{1,2})
=
\left(\sqrt{B(x)}\pm\sqrt{B(y)}\right)^{2}.
\]Thus \(B\) linearizes the two-valued law in the same way that the ordinary logarithm linearizes a one-valued formal group law. When $\mathbb G(R)$ is obtained by the modulus square construction from a formal group law with logarithm \(g(u)\) and formal inverse \(\bar u=-u\), this canonical series is precisely
\[
B(x)=g(u)g(\bar u)=-g(u)^{2},
\qquad
x=u\bar u=-u^{2}.
\]This connects the earlier definition of \(B(x)\) with the strong-isomorphism formula already used in Proposition \ref{z_1_2}.

 \begin{theorem}[{\cite[Theorem 6.4, part (1)]{Buchstaber75}}]\label{two_valued_group_differential_equation}
Let
\begin{equation}\label{general_two-valued_law}
x\ast y = \{z\mid z^2 - \Psi_1(x, y)z + \Psi_2(x, y) = 0\}
\end{equation}
be an arbitrary commutative two-valued formal group $\G(R)$ in formal power series over an arbitrary commutative $\mathbb Q$-algebra $R$ $($not necessarily an integral domain$)$. Define $B(x)$ to be series satisfying the differential equation
$$
\frac{1}{2}\phi_1(x)B'(x) + \frac{1}{8}\phi_2(x)B''(x) = 1
$$
with the initial condition $B(0)=0$, where
$$
\begin{aligned}
\phi_1(x) &= \left.\frac{\partial\Psi_1(x,y)}{\partial y}\right|_{y=0},\quad
\phi_2(x) = \left.\frac{\partial\sigma(x,y)}{\partial y}\right|_{y=0},\quad
\sigma(x,y) = \Psi_1^2 - 4\Psi_2.
\end{aligned}
$$
If $\G(R)$ is of the first type, that is, $\Psi_2(x,y)\equiv (x-y)^2 \mod \deg 3,$ then the series $B(x)$ establishes a strong isomorphism of this two-valued formal group with the elementary two-valued formal group defined by the polynomial
$$
z^2-2(x+y)z+(x-y)^2.
$$
Moreover, in this first-type case, we have:
\begin{equation}\label{B_of_x_and_phi_2}
B(x) = \left(\int\limits_{0}^{\sqrt{x}} \frac{\mathrm{d} t}{\sqrt{\phi(t^2)}} \right)^2,
\quad
\phi_2(x) = 8\int\limits_{0}^{x}\phi_1(t)\,\mathrm{d} t,
\end{equation}
where $\phi(t)=\phi_2(t)/(16t)$ and $\phi(0)=1$.
\end{theorem}

\section{The Buchstaber Genus}\label{The_Buchstaber_Genus}

A central result of this paper is the following

\begin{theorem}\label{Buchstaber_genus_theorem}
\leavevmode
\begin{enumerate}[{\bf (i)}]

\item Let
\begin{equation}\label{Integral_I_formula}
\begin{aligned}
u := I(x)=\int\limits_{0}^{\sqrt{x}}\frac{\dd t}{\sqrt{1+a_1t^2+a_2t^4+a_3t^6}},\\
g_2 = 4\left (\frac{a_1^2}{3}-a_2 \right ), \qquad
g_3=4\left(\frac{a_1a_2}{3}-\frac{2a_1^3}{27}-a_3 \right).
\end{aligned}
\end{equation}
Then
$$
x(u) = \frac{1}{\wp(u;g_2,g_3)-a_1/3}.
$$
$($If $g_2^3-27g_3^2 = 0$, then the function $
\wp(u; g_2, g_3) = -\frac{\dd^2}{\dd u^2}\log\sigma(u;g_2,g_3)$
corresponds to a degeneration of the Weierstrass $\sigma$-function$)$.

\item The logarithm and exponential of the formal two-valued group $\G_{\bm{a}}$ are given by the following formal series over the algebra $\Q\left[a_1, a_2,a_3\right]$:
\[
\begin{aligned}
B(x)
&= I^2(x)
 =\left(\int\limits_{0}^{\sqrt{x}}
 \frac{\dd t}{\sqrt{1+a_1t^2+a_2t^4+a_3t^6}}\right)^2,\\
B^{-1}(x)
&=\frac{1}{\wp(\sqrt{x};g_2,g_3)-a_1/3}.
\end{aligned}
\]

\item Let $Q(t) := 1 - a_1 t^2 + a_2 t^4 - a_3 t^6$. Then, over $\Q[a_1,a_2,a_3]$, the law $F_{\Bc}(u, v)$, its exponential $f_{\Bc}(u)$, and its logarithm $g_{\Bc}(u)$ have the form:
$$
F_{\Bc}(u,v)= \left[\left(\frac{u\sqrt{Q(v)}-v\sqrt{Q(u)}}{u^2-v^2}\right)^2+a_3u^2v^2\right]^{-1/2},
$$
\[
f_{\Bc}(u)=\frac{1}{\sqrt{\wp(u;g_2,-g_3)+\frac{a_1}{3}}},\qquad
g_{\Bc}(u) = \int\limits_{0}^{u}\frac{\dd t}{\sqrt{1-a_1t^2+a_2t^4-a_3t^6}}.
\]

\item The formal group law \(F_{\mathrm{Bc}}(u,v)\) is universal, over commutative \(\mathbb Q\)-algebras, among single-valued formal group laws with formal inverse \(\bar u=-u\) whose modulus square construction yields the two-valued formal group \(\mathbb G_{\boldsymbol a}\) with law \eqref{B_a_polynomial}, in the sense of Proposition \ref{two_valued_formal_group_roots}.

\item  The intersection of the classes $\Ff(u, v)$ and $F_{\Bc}(u,v)$ coincides with the Ochanine genus. This is characterized by the conditions
$\lambda_1(u)\equiv 1$ for $\Ff(u,v)$ and $a_3=0$ for $F_{\Bc}(u,v)$.

\end{enumerate}
\end{theorem}

\begin{proof}

\begin{enumerate}[{\bf (i)}]

\item is obtained by a series of substitutions
$$
u=t^2\to u=\frac{1}{v}\to v=w-\frac{a_1}{3}
$$
in the integral $I(x)$.

\item follows from the first part of Theorem 6.4 in \cite{Buchstaber75} and from {\bf (i)}. More concretely, rewrite the two-valued law \eqref{B_a_polynomial} in the form \eqref{general_two-valued_law}, where $$\begin{aligned}
\Psi_1(x,y)&=
\frac{
2x+2y+4a_1xy+2a_2x^2y+2a_2xy^2+4a_3x^2y^2
}{
1-2a_2xy-4a_3x^2y-4a_3xy^2+a_2^2x^2y^2-4a_1a_3x^2y^2
},\\
\Psi_2(x,y)&=
\frac{x^2-2xy+y^2}{
1-2a_2xy-4a_3x^2y-4a_3xy^2+a_2^2x^2y^2-4a_1a_3x^2y^2
}.\end{aligned}$$ Direct computations give $$\phi_1 = 2 (1 + 2 a_1 x + 3 a_2 x^2 + 4 a_3 x^3),\qquad \phi_2(x) = 16 x (1 + a_1 x + a_2 x^2 + a_3 x^3).$$ The desired formula for $B(x)$ follows from formula \eqref{B_of_x_and_phi_2}.

\item According to the modulus square construction (see Section \ref{modulus_square_construction}),
$$
B(x) = g(u)g(\overline{u}) = -g(u)^2,
$$
where $x=u\overline{u}$ and $g(u)$ is the logarithm of some formal group law $F(u, v)$.
Assume that $\overline{u} = -u$, or, equivalently, $g(-u)=-g(u)$.
Then $u = \sqrt{-x}$. From this and part {\bf (i)} it follows that
$$
g(u) = \int\limits_{0}^{u}Q(t)^{-1/2}\dd t,
$$
where $Q(t) := 1 - a_1 t^2 + a_2 t^4 - a_3 t^6$.
Applying part {\bf (i)} to the case $u:=I(w^2)$ gives the exponential $w(u)=f_{\Bc}(u)$.

Now let us derive the law $F_{\Bc}(u, v)$.
Let $\wpp(z) := \wp(z; g_2, -g_3)$, $u=f_{\Bc}(z)$, $v=f_{\Bc}(w)$.
Then
$$
\wpp(z) = \frac{1}{u^2}-\frac{a_1}{3},
\qquad
\wpp'(z)=-\frac{2\sqrt{Q(u)}}{u^3}.
$$
Substituting these formulas into the addition law
$$
\wpp(z+w) = -\wpp(z)-\wpp(w)+\frac{1}{4}\left(\frac{\wpp'(z)-\wpp'(w)}{\wpp(z)-\wpp(w)}\right)^2
$$
and carrying out straightforward computations, we obtain the law $F_{\Bc}(u, v)$.

\item follows from the previous points.

\item Assume that $F_{\Bc}(u, v) = \Ff(u, v)$ for some series $\lambda_1$, $\lambda_2$.
Since $f_{\Bc}$ is an odd function, we must require that
$$
f_{\Kr}(u) = \frac{e^{\alpha u}}{\Phi(u, z)}
$$
be odd. Hence
$$
V:= \frac{\sigma(z-u)}{\sigma(z+u)}e^{2(\zeta(z)-\alpha)u}\equiv1.
$$
From
$$
\left.\frac{\partial \log V}{\partial u}\right|_{u=0}=0
\qquad \text{and} \qquad
\left.\frac{\partial^3 \log V}{\partial u^3}\right|_{u=0}=0
$$
we obtain $\alpha=0$ and $\wp'(z)=0$.
It is known (see, for example, \cite{Krichever80}) that
$$
\Phi(u, z)\Phi(-u, z) = \wp(z) - \wp(u).
$$
Hence, from the oddness condition, we obtain:
$$
f^2_{\Kr}(u)=\frac{1}{\Phi^2(u, z)} = -\frac{1}{\wp(z) - \wp(u)}.
$$
Let
$$
\wp'(u)^2=4(\wp(u)-e_1)(\wp(u)-e_2)(\wp(u)-e_3).
$$
Then from
$$
f'_{\Kr}(u)^2=
\left(1+(e_1-e_2)f_{\Kr}(u)^2\right)
\left(1+(e_1- e_3)f_{\Kr}(u)^2\right)
$$
and
$$
f_{\Bc}'(u)^2=1-a_1f_{\Bc}(u)^2+a_2f_{\Bc}(u)^4-a_3f_{\Bc}(u)^6
$$
we get $a_3=0$.

\end{enumerate}
\end{proof}

Recall
\begin{equation}\label{R_Bc}
R_{\Bc}:=\Z\!\left[\frac{1}{2},a_1,a_2,a_3\right],\quad \deg a_i=-4i,\qquad i=1,2,3.
\end{equation}

\begin{coroll}\label{coroll_Buchstaber_genus}
The formal group law $F_{\Bc}(u,v)$ corresponds to a Hirzebruch genus
$$
\Bc:\Omega_{\Uu}\to R_{\Bc},
$$
which we call the Buchstaber genus.
\end{coroll}

\begin{proof}
The formula in Theorem {\rm\ref{Buchstaber_genus_theorem}(iii)} defines a formal group law
\[
F_{\Bc}(u,v)\in R_{\Bc}[[u,v]].
\]

Since $Q(0)=1$ and $2$ is invertible in $R_{\Bc}$, we have $\sqrt{Q(t)}\in R_{\Bc}[[t]]$. The numerator in
\[
H(u,v):=\frac{u\sqrt{Q(v)}-v\sqrt{Q(u)}}{u-v}
\]
is divisible by $u-v$, so $H\in R_{\Bc}[[u,v]]$ and $H(0,0)=1$. Choosing the branch with leading term $u+v$, the formula of Theorem \ref{Buchstaber_genus_theorem}{\rm (iii)} becomes
\[
F_{\Bc}(u,v)
 =\frac{u+v}{\sqrt{H(u,v)^2+a_3u^2v^2(u+v)^2}}.
\]
Thus \(F_{\mathrm{Bc}}(u,v)\in R_{\mathrm{Bc}}[[u,v]]\). By the universality of the formal group law of complex cobordism, equivalently by the identification of \(\Omega_{\Uu}\) with the Lazard ring, \(F_{\mathrm{Bc}}\) determines a unique graded ring homomorphism
\[
\mathrm{Bc}\colon \Omega_{\Uu}\longrightarrow R_{\mathrm{Bc}}.
\]After extension of scalars to \(\mathbb Q[a_1,a_2,a_3]\), the associated formal group law has logarithm \(g_{\mathrm{Bc}}\) and exponential \(f_{\mathrm{Bc}}\) from Theorem \ref{Buchstaber_genus_theorem}{\rm (iii)} . Hence this homomorphism is precisely the Buchstaber genus. In particular, the Buchstaber genus takes values in \(R_{\mathrm{Bc}}\).
\end{proof}

\begin{coroll}\label{coroll_integral_universality}
The formal group law $F_{\Bc}(u,v)$ over $R_{\Bc}$ is universal among
one-dimensional commutative formal group laws with formal inverse
$\bar u=-u$ whose modulus square construction belongs to the
Buchstaber family. Namely, if $S$ is a commutative
$\Z[1/2]$-algebra and $F(u,v)$ is such a formal group law over $S$,
then there exists a unique $\Z[1/2]$-algebra homomorphism
\[
\varphi\colon R_{\Bc}\longrightarrow S
\]
such that
\[
F(u,v)=\varphi_*F_{\Bc}(u,v).
\]  where \(\varphi_*F_{\Bc}\in S[[u,v]]\) denotes the formal group law obtained from \(F_{\Bc}\) by applying \(\varphi\) coefficientwise.
\end{coroll}

\begin{proof}
The law $F_{\Bc}(u, v)$ is a formal group law over $\Q[a_1,a_2,a_3]$ by Theorem \ref{Buchstaber_genus_theorem}{\rm (iii)}. Since $R_{\Bc}$ is a subring of $\Q[a_1,a_2,a_3]$, all formal group identities already hold over $R_{\Bc}$. We now compute its modulus square. Put
\[
x=-u^2,\qquad y=-v^2,\qquad P(t)=1+a_1t+a_2t^2+a_3t^3,
\]
so that $Q(u)=P(x)$ and $Q(v)=P(y)$, and define
\[
\begin{aligned}
D&=(x-y)^2,\\
M&=xP(y)+yP(x)-a_3xyD
   =x+y+2a_1xy+a_2xy(x+y)+2a_3x^2y^2,\\
T&=2uv\sqrt{P(x)P(y)},\\
C&=(1-a_2xy)^2-4a_3xy(x+y+a_1xy).
\end{aligned}
\]
For
\[
z_+=-F_{\Bc}(u,v)^2,
\qquad
z_-=-F_{\Bc}(u,-v)^2,
\]
the defining formula gives, in the fraction field,
\[
z_\pm=\frac{D}{M\pm T}.
\]
A direct simplification yields $M^2-T^2=DC$, and hence
\[
z_++z_-=\frac{2M}{C},
\qquad
z_+z_-=\frac{D}{C}.
\]
Therefore $z_+$ and $z_-$ satisfy $Cz^2-2Mz+D=0$. On the other hand, expansion of \eqref{B_a_polynomial} as a quadratic polynomial in $z$ gives the identity
\[
\Bb_{\boldsymbol a}(z;x,y)=Cz^2-2Mz+D.
\]
Thus the modulus square of $F_{\Bc}$ is exactly $\G_{\boldsymbol a}$ over $R_{\Bc}$.

Now, let $S$ be a commutative $\Z[1/2]$-algebra, and let $F(u,v)$ be a
formal group law over $S$ with formal inverse $-u$ whose modulus square
is a specialization $\mathbb G_{\boldsymbol\alpha}$ of the
Buchstaber family, where $\alpha_1,\alpha_2,\alpha_3\in S$. Define
\[
\varphi\colon R_{\Bc}\longrightarrow S,
\qquad
\varphi(a_i)=\alpha_i.
\]
The explicit modulus square computation above shows that
$\varphi_*F_{\Bc}$ has the same modulus square as $F$.

Put
\[
q=-F(u,v)^2,\qquad
r_+=-\varphi_*F_{\Bc}(u,v)^2,\qquad
r_-=-\varphi_*F_{\Bc}(u,-v)^2.
\]
Then $q,r_+,r_-$ are roots, in the indicated pairs, of the same
quadratic Buchstaber polynomial (see \eqref{two-valued_law}). Hence
\[
(q-r_+)(q-r_-)=0.
\]
Since both formal group laws have unit $0$, the series $q-r_-$
vanishes after setting either $u=0$ or $v=0$. Therefore
\[
q-r_-=uv\,U(u,v)
\]
for some $U\in S[[u,v]]$. Its lowest homogeneous term is $-4uv$, so
$U(0,0)=-4$. Since $2$ is invertible in $S$, the series $U$ is a
unit. Multiplication by $uv$ is injective in $S[[u,v]]$, and hence
$q-r_-$ is a non-zero-divisor. Consequently $q=r_+$, that is,
\[
F(u,v)^2=\varphi_*F_{\Bc}(u,v)^2.
\]

Because the formal inverse of both laws is $-u$, both series are
divisible by $u+v$:
\[
F(u,v)=(u+v)A(u,v),\qquad
\varphi_*F_{\Bc}(u,v)=(u+v)B(u,v),
\]
where $A(0,0)=B(0,0)=1$. It follows that $A^2=B^2$. Since $A+B$ has
constant term $2$ and is therefore a unit, we obtain $A=B$. Thus
\[
F=\varphi_*F_{\Bc}.
\]

Finally, the parameters $\alpha_1,\alpha_2,\alpha_3$ are uniquely
determined by the coefficients of the Buchstaber polynomial, so the
homomorphism $\varphi$ is unique.
\end{proof}

In \cite{Buchstaber_Veselov_24}, it was established that the exponential of the universal formal group law of complex cobordism is given by the series
$$
f_{\Uu}(u)=u+\sum_{n\ge 1}[\Theta_n]\frac{u^{n+1}}{(n+1)!},
$$
where $[\Theta_n]$ denotes the cobordism class of a smooth theta divisor (of complex dimension $n$) on a general principally polarised abelian variety $\mathcal A_{n+1}$. 

It is well known that $\Omega_{\Uu}\otimes\Q\cong\Q[[\CP^1], [\CP^2],..., [\CP^n],...]$, see \cite{Novikov62}. In \cite[formula (4)]{Buchstaber_Veselov26} the following identity in $\Omega_{\Uu}\otimes\Q$ was established: $$[\CP^n] = (n+1)\Lll_n(\tau_1,...,\tau_n),\quad \tau_k = \frac{[\Theta_k]}{(k+1)!},$$ where $\Lll_n$ denotes the Lagrange inversion polynomial. Hence, $$\Omega_{\Uu}\otimes\Q\cong\Q[[\Theta_1],...,[\Theta_n],...].$$ The ring $\Z[[\Theta_1],...,[\Theta_n],...]$ is a proper subring of $\Omega_{\Uu}$. By a famous result of Milnor and Novikov \cite{Milnor60,Novikov60}, $\Omega_{\Uu}$ is a polynomial ring $\Z[u_1,\ldots,u_n,\ldots]$ with one generator $u_n\in\Omega_{\Uu}^{-2n}$ in every dimension $2n$ (where $n\geq 1$). The problem of finding good geometric representatives for the generators $u_n$ remains open.   

From the universality of the exponential $f_{\Uu}$, we obtain the identity:
$$
f_{\Bc}(u)=u+\sum_{n\ge 1}\Bc(\Theta_n)\frac{u^{n+1}}{(n+1)!}.
$$
From the oddness of $f_{\Bc}(u)$ it follows that
$$
\Bc(\Theta_{2n+1}) = 0 \qquad \text{for all } n\geq 0.
$$
We have:
\begin{equation}\label{Theta_2_4_6}
\Bc(\Theta_2)=-a_1,\quad
\Bc(\Theta_4)=a_1^2+12a_2,\quad
\Bc(\Theta_6)=-a_1^3-132a_1a_2-360a_3.
\end{equation}

Observe that the first three nonzero values $\Bc(\Theta_n)$ lie in the ring $\Z[a_1,a_2,a_3]$. It is not a coincidence.

\begin{prop}\label{Bc_of_Theta_is_in_Z}
The values $\Bc(\Theta_n)$ belong to the ring $\Z[a_1,a_2,a_3]$.
\end{prop}

\begin{proof}
For brevity, write $f:=f_{\Bc}$. From Theorem \ref{Buchstaber_genus_theorem}, point {\bf (iii)}, we know that \begin{equation}\label{f_prime_series}f'(x)=\sqrt{1-a_1f(x)^2+a_2f(x)^4-a_3f(x)^6}.\end{equation} Squaring and differentiating \eqref{f_prime_series} gives: \begin{equation}\label{f_two_primes}f''=-a_1f+2a_2f^3-3a_3f^5, \quad \text{where } f(0)=0,\ f'(0)=1.\end{equation} Denote $b_n=\Bc(\Theta_{2n})$, $n>0$ and $b_0=1$. Then we have: \begin{equation}\label{f_series}f(x)=\sum_{n\ge0} b_n\,\frac{x^{2n+1}}{(2n+1)!}.\end{equation} Using \eqref{f_two_primes} and \eqref{f_series}, the coefficient comparison implies $$\begin{aligned}
b_{n+1}
&=-a_1 b_n \\
&\quad+2a_2
\sum_{i+j+k=n-1}
\binom{2n+1}{2i+1,\,2j+1,\,2k+1} b_i b_j b_k \\
&\quad-3a_3
\sum_{i_1+\cdots+i_5=n-2}
\binom{2n+1}{2i_1+1,\dots,2i_5+1}
\prod_{r=1}^5 b_{i_r}, \ n\geq 2.
\end{aligned}$$ Hence, $b_n\in\Z[a_1,a_2,a_3]$.
\end{proof}

Now we compute the values of the Buchstaber genus on $\CP^n$'s. Applying the genus to Mishchenko’s logarithm for the formal group law $F_{\Uu}(u, v)$, we get:
$$
g_{\Bc}(u)=\sum_{n\ge 0}\Bc(\CP^n)\frac{u^{n+1}}{n+1}.
$$
Hence,
$$
\Bc(\CP^{2n+1})=0 \qquad \text{for all } n\geq 0.
$$
We obtain:
$$
\Bc(\CP^2) = \frac{a_1}{2},\quad
\Bc(\CP^4) = \frac{3a_1^2-4a_2}{8},\quad
\Bc(\CP^6) =\frac{5a_1^3-12a_1a_2+8a_3}{16}.
$$

Recall that the Krichever genus corresponds to a formal group law with the exponential \eqref{f_Krichever} given by the Baker–Akhiezer function $\Phi(x, z)$. This genus depends on 4 parameters: $$b_1 = \alpha,\quad b_2=\wp(z),\quad b_3=\wp'(z),\quad \text{and}\quad b_4 = g_2.$$ As it was proved in \cite{Buchstaber90}, any formal group law of the form \eqref{the_class_F} has $f_{\Kr}$ as its exponential.  

The Krichever genus $\Kr:\Omega_{\Uu}\to A$, where $A$ is the ring of coefficients of the series $\lambda_1$, $\lambda_2$ from \eqref{the_class_F}, is an isomorphism of graded abelian groups in real dimensions less than $10$ \cite[Corollary 6.11]{BPR10}.

There exists a class \(K\in \Omega_U^{-12}\) on which the Krichever genus tensored by $\Q$ vanishes, while the Buchstaber genus remains nonzero, even for \(a_1=a_2=0\). By \eqref{Theta_2_4_6}, degree \(12\) is minimal for the existence of such a class.

\begin{prop}\label{element_K}
\leavevmode
\begin{enumerate}[\bf (i)]
\item Let $\Theta_n := [\Theta_n]$ denote the class of a smooth theta divisor of complex dimension $n$ in $\Omega_{\Uu}^{-2n}$. Introduce an element $K\in\Omega_{\Uu}^{-12}$$:$
\begin{equation}\label{K_generator}
\begin{aligned}
K={}&\Theta_6
+9\Theta_1^6
-15\Theta_1^4\Theta_2
-3\Theta_1^3\Theta_3
-13\Theta_1^2\Theta_2^2
+3\Theta_1^2\Theta_4 \\
&\quad
+29\Theta_1\Theta_2\Theta_3
+10\Theta_2^3
-11\Theta_2\Theta_4
-10\Theta_3^2.
\end{aligned}
\end{equation}
Then
\[
\Bc(K)=-360a_3\in\Z[a_1,a_2,a_3],
\qquad
\Kr(K)=0\in A\otimes\Q
\cong\Q[\alpha,\wp(z),\wp'(z),g_2],
\]
where $A$ is the ring of coefficients of the series $\lambda_1$, $\lambda_2$ from \eqref{the_class_F}. 

\item The $\Q$-vector space $$\Ker(\Kr:\Omega^{-12}_{\Uu}\otimes\Q\to \Q[\alpha, \wp(z), \wp'(z), g_2])$$ is $2$-dimensional and generated by $K$ and $$L=
\Theta_5\Theta_1
-3\Theta_4\Theta_1^2
-11\Theta_3\Theta_2\Theta_1
+12\Theta_3\Theta_1^3
+22\Theta_2^2\Theta_1^2
-30\Theta_2\Theta_1^4
+9\Theta_1^6.$$
\end{enumerate}
\end{prop} 

\begin{proof}
\begin{enumerate}[\bf (i)]
\item Let $u=\wp(z)$ and $v=\wp'(z)$. From the expansion $f_{\Kr}(u)$, we obtain: $$\begin{aligned}
\Kr(\Theta_1)&=2\alpha,\\
\Kr(\Theta_2)&=3(\alpha^2+u),\\
\Kr(\Theta_3)&=4(\alpha^3+3\alpha u-v),\\
\Kr(\Theta_4)&=5\alpha^4+30\alpha^2u-20\alpha v+45u^2-3g_2,\\
\Kr(\Theta_5)&=6\alpha^5+60\alpha^3u-60\alpha^2v+270\alpha u^2-18\alpha g_2-132uv,\\
\Kr(\Theta_6)&=
7\alpha^{6}
-63\alpha^{2}g_{2}
+105\alpha^{4}u
-99g_{2}u
+945\alpha^{2}u^{2}\\
&+1215u^{3}
-140\alpha^{3}v
-924\alpha uv
+160v^{2}.
\end{aligned}$$ Combining it with the earlier computations \eqref{Theta_2_4_6}, we find: $\Bc(K)=-360a_3$ and $\Kr(K)=90\bigl(4u^3-g_2u-g_3-v^2\bigr)=0$.

\item This part follows from direct computations.

\end{enumerate}
\end{proof}

\section{The Buchstaber Cohomology Theory}\label{Buchstaber_cohomology_theory}

In this section, all generalized cohomology theories are considered on the category of finite CW-complexes. Recall that
\[
R_{\Bc}:=\Z\!\left[\frac12,a_1,a_2,a_3\right],
\qquad \deg a_i=-4i.
\]

\begin{theorem}\label{Buc_spectrum}
There exists a complex-oriented multiplicative cohomology theory $\Buc^{\ast}$ with a commutative and associative product such that
\[
\Buc^{\ast}(\pt)\cong R_{\Bc},
\]
and the coefficient homomorphism induced by the complex orientation,
\[
\Omega_{\Uu}=\Uu^{\ast}(\pt)\longrightarrow \Buc^{\ast}(\pt),
\]
coincides with the Buchstaber genus $\Bc:\Omega_{\Uu}\to R_{\Bc}$.
\end{theorem}

\begin{proof}
We use the Baas--Sullivan construction in oriented bordism after inverting $2$, following the method of Landweber, Ravenel, and Stong \cite{Baas73,LRS95}. The logarithm $g_{\Bc}$ and the exponential $f_{\Bc}$ are odd. Hence the characteristic series
\[
Q_{\Bc}(u)=\frac{u}{f_{\Bc}(u)}
\]
is even, and the Buchstaber genus factors through oriented bordism:
\[
\Omega_{\Uu}\longrightarrow
\Omega_{\SO}\!\left[\frac12\right]
\xrightarrow{\,\varphi_{\Bc}\,}
R_{\Bc}.
\]

Recall that \cite[Section~3]{LRS95}
\[
\Omega_{\SO}\!\left[\frac12\right]
\cong
\Z\!\left[\frac12\right][x_4,x_8,x_{12},x_{16},\ldots].
\]
Using the values of $\Bc$ on complex projective spaces computed above, choose the first three polynomial generators to be
\[
\begin{aligned}
x_4&=2[\CP^2],\\
x_8&=-2[\CP^4]+3[\CP^2]^2,\\
x_{12}&=2[\CP^6]-6[\CP^2][\CP^4]+4[\CP^2]^3.
\end{aligned}
\]
We have
\[
\begin{aligned}
\varphi_{\Bc}(x_4)
&=2\cdot\frac{a_1}{2}=a_1,\\
\varphi_{\Bc}(x_8)
&=-2\cdot\frac{3a_1^2-4a_2}{8}
  +3\left(\frac{a_1}{2}\right)^2=a_2,\\
\varphi_{\Bc}(x_{12})
&=2\cdot\frac{5a_1^3-12a_1a_2+8a_3}{16}
 -6\cdot\frac{a_1}{2}\cdot\frac{3a_1^2-4a_2}{8}
 +4\left(\frac{a_1}{2}\right)^3=a_3.
\end{aligned}
\]
Extend these elements to a set of polynomial generators $x_4,x_8,x_{12},x_{16},\ldots$. For every $n>3$, let $P_{4n}(\alpha,\beta,\gamma)$ be the homogeneous polynomial satisfying
\[
\varphi_{\Bc}(x_{4n})=P_{4n}(a_1,a_2,a_3),
\]
and put
\[
y_{4n}=x_{4n}-P_{4n}(x_4,x_8,x_{12}).
\]
The triangular change of generators gives
\[
\Omega_{\SO}\!\left[\frac12\right]
\cong
\Z\!\left[\frac12\right]
[x_4,x_8,x_{12},y_{16},y_{20},\ldots],
\]
and $\Sigma=\{y_{16},y_{20},\ldots\}$ is a regular sequence. Multiplying its elements by suitable powers of $2$, if necessary, does not change the generated ideal and allows us to choose closed oriented representatives for the singularity set.

The Baas--Sullivan theory of oriented bordism with singularities $\Sigma$ therefore has coefficient ring
\[
\Buc^{\ast}(\pt)
:=
\SO_{\Sigma}^{\ast}(\pt)\!\left[\frac12\right]
\cong
\frac{\Omega_{\SO}[1/2]}{(y_{16},y_{20},\ldots)}
\cong
R_{\Bc}.
\]
Since $2$ is invertible, Mironov's obstruction theory gives a commutative and associative multiplication on this theory \cite{Mironov79}. Finally, the composite of ring spectra
\[
\MU\longrightarrow \mathrm{MSO}\!\left[\frac12\right]
\longrightarrow \Buc
\]
provides a complex orientation. By construction, its homomorphism on coefficients is $\varphi_{\Bc}$ composed with the forgetful map, hence it is precisely the Buchstaber genus.
\end{proof}

Put
\begin{equation}\label{Buchstaber_discriminant}
\begin{aligned}
\Delta
&=\disc_t\bigl(t^3-a_1t^2+a_2t-a_3\bigr)\\
&=a_1^2a_2^2-4a_1^3a_3-4a_2^3+18a_1a_2a_3-27a_3^2.
\end{aligned}
\end{equation}

\begin{theorem}\label{Buchstaber_spectrum_is_Landweber_exact}
The functor
\[
\Uu^{\ast}(-)\otimes_{\Omega_{\Uu}}
R_{\Bc}[\Delta^{-1}]
\]
is a complex-oriented multiplicative cohomology theory. Moreover, the natural transformation
\[
\Uu^{\ast}(-)\otimes_{\Omega_{\Uu}}
R_{\Bc}[\Delta^{-1}]
\longrightarrow
\Buc^{\ast}(-)[\Delta^{-1}]
\]
is an isomorphism.
\end{theorem}

\begin{proof}
Let $f=f_{\Bc}$. From $f'^2=1-a_1f^2+a_2f^4-a_3f^6$, set
\[
T=f^{-2},\qquad Y=\frac{dT}{du}=-2f^{-3}f'.
\]
Then
\[
Y^2=4\bigl(T^3-a_1T^2+a_2T-a_3\bigr),
\qquad
\frac{dT}{Y}=du.
\]
Thus $F_{\Bc}$ is the formal group at infinity of this cubic curve, equipped with its invariant differential. After inverting $\Delta$, the curve is a smooth elliptic curve over $R_{\Bc}[\Delta^{-1}]$.

We apply the Landweber exact functor criterion in the form used in \cite[Section~4.5]{LRS95}. The prime $2$ is already invertible. Let $p$ be odd, and write the $p$-series of the formal group as
\[
[p](z)=pz+\cdots+v_1z^p+\cdots+v_2z^{p^2}+\cdots.
\]
Multiplication by $p$ is injective on $R_{\Bc}[\Delta^{-1}]$. Modulo $p$, the ring is a domain and $v_1$ is not the zero element, because the family contains ordinary elliptic curves; hence $v_1$ is a non-zero-divisor. If $v_2$ were not a unit modulo $(p,v_1)$, some residue field would give a smooth elliptic curve whose formal group has $v_1=v_2=0$, and therefore height at least $3$. This is impossible: the formal group of an elliptic curve in positive characteristic has height $1$ or $2$. Thus $v_2$ is a unit modulo $(p,v_1)$, and the Landweber criterion is satisfied for every prime.

Consequently the extension-of-scalars functor in the statement is a cohomology theory. The orientation of $\Buc$ gives the displayed natural transformation; it is an isomorphism on coefficient rings. The comparison theorem for cohomology theories on finite CW-complexes now implies that it is an isomorphism on every finite CW-complex, exactly as in \cite[Section~4.7]{LRS95}.
\end{proof}

\begin{remark}
The localization at $\Delta$ is essential for the Landweber-exact construction. The unlocalized theory $\Buc^{\ast}$ is instead supplied by bordism with singularities.
\end{remark}

\section{Hurwitz Integrality and Bunkova's Conjecture}\label{Bunkova_conjecture}

Recall that the Witten genus \cite{Witten88} corresponds to a formal group law with the exponential $f_{\Wt}(u)=\sigma(u;g_2,g_3)$ given by the Weierstrass $\sigma$-function. For the specialization $a_1=0$, first compare the Buchstaber
\[
f_{\Bc}(u)=\frac{1}{\sqrt{\wp(u;g_2,-g_3)}}
\]
and Witten $f_{\Wt}(u)=\sigma(u;g_2,g_3)$ exponentials:
{\small
\begin{equation}\label{f_Bc_and_f_Wt}
\begin{alignedat}{7}
f_{\Bc}(u)&=u
&{}-3g_2\frac{u^5}{5!}&{}
&{}+90g_3\frac{u^7}{7!}&{}
&{}+189g_2^2\frac{u^9}{9!}&{}
&{}-43\,740g_2g_3\frac{u^{11}}{11!}&{}
&{}+\bigl(-68\,607g_2^3+2\,673\,000g_3^2\bigr)
       \frac{u^{13}}{13!}&{}
&{}+O(u^{15}),&{}\\[0.4ex]
f_{\Wt}(u)&=u
&{}-\widetilde g_2\frac{u^5}{5!}&{}
&{}-\widetilde g_3\frac{u^7}{7!}&{}
&{}-9\widetilde g_2^2\frac{u^9}{9!}&{}
&{}-6\widetilde g_2\widetilde g_3\frac{u^{11}}{11!}&{}
&{}+\bigl(69\widetilde g_2^3-6\widetilde g_3^2\bigr)
       \frac{u^{13}}{13!}&{}
&{}+O(u^{15}),&{}
\end{alignedat}
\end{equation}
}where $\widetilde g_2=g_2/2$ and $\widetilde g_3=6g_3$.

Recall that a formal power series
\[
\sum_{n\geq0}c_n\frac{u^{n}}{n!},
\qquad c_n\in R,
\]
is called a Hurwitz series over a torsion-free commutative ring $R$. Put
\begin{equation}\label{Hurwitz_series_definition}
\mathcal H_R[[u]]
=
\left\{\left.\sum_{n\geq0}c_n\frac{u^n}{n!}\in (R\otimes_{\mathbb Z}\mathbb Q)[[u]]\ \right|\ c_n\in R\right\}.
\end{equation}

In \cite{Bunkova17}, Bunkova formulated a coefficient-divisibility conjecture giving exact $2$-adic and $3$-adic valuation formulas for the normalized coefficients of the sigma-function. One consequence of this conjecture is the Hurwitz-integrality statement
\[
\sigma(u;g_2,g_3)\in\mathcal H_{\Z[g_2/2,6g_3]}[[u]].
\]
Corollary~\ref{Bunkova_Hurwitz_prediction} below proves this consequence, but not the stronger valuation formulas themselves. We first determine the minimal Hurwitz coefficient ring of the full three-parameter Buchstaber exponential and then specialize to $a_1=0$.

\begin{lemma}\label{Hurwitz_divisibility_lemma}
Let $R$ be a torsion-free commutative ring and let
\[
h(u)=\sum_{m\geq 1}h_m\frac{u^m}{m!}
\in\mathcal H_R[[u]]
\]
be a Hurwitz series with vanishing constant term. Then, for every $k\geq 0$,
\[
\frac{h(u)^k}{k!}\in\mathcal H_R[[u]].
\]
More generally, if
\[
\psi(t)=\sum_{k\geq 0}c_k\frac{t^k}{k!}
\in\mathcal H_R[[t]],
\]
then
\[
\psi(h(u))\in\mathcal H_R[[u]].
\]
The divisor $k!$ is optimal uniformly over all torsion-free commutative rings $R$ and all such $h$.
\end{lemma}

\begin{proof}
The exponential partial Bell polynomials satisfy
\[
\frac{1}{k!}
\left(
  \sum_{m\geq 1}x_m\frac{u^m}{m!}
\right)^k
=
\sum_{n\geq k}
B_{n,k}(x_1,\ldots,x_{n-k+1})
\frac{u^n}{n!},
\]
where
\[
B_{n,k}\in
\mathbb Z[x_1,\ldots,x_{n-k+1}]
\]
is the $(n,k)$-th partial Bell polynomial; see
\cite[Chapter~III, \S3.3]{Comtet74}. Substituting $x_m=h_m$
gives
\[
\frac{h(u)^k}{k!}
=
\sum_{n\geq k}
B_{n,k}(h_1,\ldots,h_{n-k+1})
\frac{u^n}{n!},
\]
and therefore $h^k/k!\in\mathcal H_R[[u]]$.

More generally,
\[
\begin{aligned}
\psi(h(u))
&=
\sum_{k\geq 0}c_k\frac{h(u)^k}{k!} \\
&=
\sum_{n\geq 0}
\left(
  \sum_{k=0}^{n}
  c_k B_{n,k}(h_1,\ldots,h_{n-k+1})
\right)
\frac{u^n}{n!}.
\end{aligned}
\]
Every coefficient in parentheses belongs to $R$, proving the
composition statement.

Finally, take $R=\mathbb Z$ and $h(u)=u$. Then
\[
h(u)^k=u^k=k!\frac{u^k}{k!}.
\]
Thus any positive integer that divides $h^k$ in the Hurwitz ring for
every such $h$ must divide $k!$. Hence $k!$ is the largest possible
universal divisor, uniformly over $R$ and $h$.
\end{proof}

\begin{theorem}\label{Buchstaber_exponential_Hurwitz_theorem}
For algebraically independent parameters $a_1,a_2,a_3$, the exponential
of the Buchstaber genus satisfies
\[
f_{\mathrm{Bc}}(u)
\in
\mathcal H_{\mathbb Z[a_1,\,12a_2,\,360a_3]}[[u]].
\]
This is the minimal Hurwitz coefficient ring of $f_{\mathrm{Bc}}$
among subrings of $\mathbb Q[a_1,a_2,a_3]$.
\end{theorem}

\begin{proof}
Put
\[
R_H:=\Z[a_1,12a_2,360a_3]
\]
and write $f=f_{\mathrm{Bc}}$. By \eqref{f_prime_series},
\[
(f')^2=1-a_1f^2+a_2f^4-a_3f^6,
\qquad f(0)=0,\quad f'(0)=1.
\]
Since $f'$ is a unit, differentiation gives
$f''=-a_1f+2a_2f^3-3a_3f^5$, as in \eqref{f_two_primes}.
In divided-power form, this becomes
\begin{equation}\label{eq:Bc_exp_Bell}
f''
=
-a_1f
+
(12a_2)\frac{f^3}{3!}
-
(360a_3)\frac{f^5}{5!}.
\end{equation}

Write
\[
f(u)=\sum_{n\geq0}c_n\frac{u^n}{n!},
\qquad c_0=0,\quad c_1=1.
\]
We use the exponential partial Bell polynomials from the proof of
Lemma~\ref{Hurwitz_divisibility_lemma}, with the convention $B_{n,k}=0$
for $n<k$. Comparing the $n$-th Hurwitz coefficients in
\eqref{eq:Bc_exp_Bell} gives
\[
\begin{aligned}
c_{n+2}
&=-a_1c_n
  +12a_2 B_{n,3}(c_1,\ldots,c_{n-2})\\
&\quad-360a_3 B_{n,5}(c_1,\ldots,c_{n-4}),
\qquad n\geq0.
\end{aligned}
\]
In particular, $c_2=0$. Each Bell polynomial has integer coefficients,
and the right-hand side involves only $c_j$ with $j\leq n$.
Starting with $c_0,c_1\in R_H$, induction therefore gives
$c_n\in R_H$ for every $n\geq0$. Thus
$f_{\mathrm{Bc}}\in\mathcal H_{R_H}[[u]]$.

To prove minimality, the first relevant coefficients are
\[
\begin{aligned}
c_3&=-a_1,\\
c_5&=a_1^2+12a_2,\\
c_7&=-a_1^3-132a_1a_2-360a_3,
\end{aligned}
\]
in agreement with the theta-divisor values in \eqref{Theta_2_4_6}.
Every Hurwitz coefficient ring for $f_{\mathrm{Bc}}$ must contain
$c_3,c_5,c_7$, and hence also
\[
\begin{aligned}
a_1&=-c_3,\\
12a_2&=c_5-a_1^2,\\
360a_3&=-c_7-a_1^3-11a_1(12a_2).
\end{aligned}
\]
Consequently it contains $R_H$, which proves the asserted minimality.
\end{proof}

By the theta-divisor expansion of $f_{\Bc}$ in
Section~\ref{The_Buchstaber_Genus}, the theorem also identifies
\[
\Z\bigl[\Bc(\Theta_n)\mid n\geq1\bigr]
=
\Z[a_1,12a_2,360a_3],
\]
refining Proposition~\ref{Bc_of_Theta_is_in_Z}.
This is the Hurwitz coefficient ring of $f_{\Bc}$, not the coefficient
ring $R_{\Bc}$ of the genus on all of $\Omega_{\Uu}$.

\begin{coroll}\label{Buchstaber_exponential_Hurwitz_specialization}
For $a_1=0$, the exponential of the Buchstaber genus satisfies
\[
f_{\mathrm{Bc}}(u)
=
\frac{1}{\sqrt{\wp(u;g_2,-g_3)}}
\in
\mathcal H_{\mathbb Z[3g_2,\,90g_3]}[[u]].
\]
Moreover, $\mathbb Z[3g_2,90g_3]$ is its minimal Hurwitz coefficient
ring among subrings of $\mathbb Q[g_2,g_3]$.
\end{coroll}

\begin{proof}
At $a_1=0$, the parameter relations \eqref{Integral_I_formula} give
\[
a_2=-\frac{g_2}{4},\qquad a_3=-\frac{g_3}{4}.
\]
Thus $12a_2=-3g_2$ and $360a_3=-90g_3$, so the specialization of
$\Z[a_1,12a_2,360a_3]$ is $\Z[3g_2,90g_3]$.
The integrality assertion follows from
Theorem~\ref{Buchstaber_exponential_Hurwitz_theorem}, and the
Weierstrass expression for $f_{\Bc}$ follows from
Theorem~\ref{Buchstaber_genus_theorem}{\rm (iii)}.
Finally, the expansion \eqref{f_Bc_and_f_Wt} begins with
\[
f_{\Bc}(u)
=
u-3g_2\frac{u^5}{5!}+90g_3\frac{u^7}{7!}+O(u^9).
\]
Hence every Hurwitz coefficient ring contains $3g_2$ and $90g_3$,
proving minimality also after specialization.
\end{proof}

\begin{coroll}[(Bunkova's Hurwitz-integrality conjecture)]\label{Bunkova_Hurwitz_prediction}
The Weierstrass sigma-function satisfies
\[
\sigma(u;g_2,g_3)\in\mathcal H_{\Z[g_2/2,6g_3]}[[u]],
\]
and $\Z[g_2/2,6g_3]$ is its minimal Hurwitz coefficient ring among subrings of $\Q[g_2,g_3]$.
\end{coroll}

\begin{proof}
Put
\[
A:=\Z\!\left[\frac{g_2}{2},6g_3\right]
\]
and set
\[
f(u):=\frac{1}{\sqrt{\wp(u;g_2,g_3)}}.
\]
Applying Corollary~\ref{Buchstaber_exponential_Hurwitz_specialization} with $g_3$ replaced by $-g_3$, we obtain
\[
f\in\mathcal H_{\Z[3g_2,90g_3]}[[u]]
\subseteq
\mathcal H_A[[u]].
\]

Since both $\sigma(u;g_2,g_3)$ and $f(u)$ are odd and have linear
term $u$, their quotient is an even unit with constant term $1$.
Therefore
\[
L(u):=\log\frac{\sigma(u;g_2,g_3)}{f(u)}
\]
is even, and hence
\[
L(0)=L'(0)=0.
\]

Since $\wp=f^{-2}$, the Weierstrass equation gives
\[
(f')^2
=
1-\frac{g_2}{4}f^4-\frac{g_3}{4}f^6,
\qquad
f''
=
-\frac{g_2}{2}f^3-\frac{3g_3}{4}f^5.
\]
Using
\[
\frac{\sigma'}{\sigma}=\zeta,
\qquad
\zeta'=-\wp=-\frac{1}{f^2},
\]
we obtain
\[
\begin{aligned}
L''
&=
-\frac{1}{f^2}-\frac{f''}{f}
+\frac{(f')^2}{f^2}  \\
&=
\frac{g_2}{4}f^2+\frac{g_3}{2}f^4 \\
&=
\left(\frac{g_2}{2}\right)\frac{f^2}{2!}
+
2(6g_3)\frac{f^4}{4!}.
\end{aligned}
\]
Because $f\in\mathcal H_A[[u]]$ has vanishing constant term,
Lemma~\ref{Hurwitz_divisibility_lemma} implies that $f^2/2!$ and $f^4/4! \in\mathcal H_A[[u]]$. Thus
\[
L''\in\mathcal H_A[[u]].
\]
Since integration preserves Hurwitz series and
$L(0)=L'(0)=0$, it follows that
\[
L\in\mathcal H_A[[u]].
\]

Applying the composition statement of Lemma~\ref{Hurwitz_divisibility_lemma} to
$\psi(t)=e^t$, we obtain
\[
e^{L(u)}\in\mathcal H_A[[u]].
\]
Consequently,
\[
\sigma(u;g_2,g_3)
=
f(u)e^{L(u)}
\in\mathcal H_A[[u]].
\]

Finally,
\[
\sigma(u;g_2,g_3)
=
u-\frac{g_2}{2}\frac{u^5}{5!}
-6g_3\frac{u^7}{7!}
+\cdots.
\]
Hence every Hurwitz coefficient ring for the sigma-function contained
in $\Q[g_2,g_3]$ must contain both $g_2/2$ and $6g_3$. Therefore the ring $\Z\!\left[g_2/2,6g_3\right]$ is the minimal such ring.
\end{proof}

\printbibliography[heading=bibintoc]

\vspace{1.5em}
\noindent
\textit{Mikhail Kornev}\\
Steklov Mathematical Institute of Russian Academy of Sciences\\
\begin{tabular}{@{}l@{\ }l@{}}
Email: & \texttt{mkorneff@mi-ras.ru}
\end{tabular}

\vfill
\begin{center}
  \pgfornament[width=2.2cm]{188}
\end{center}

\end{document}